\documentclass[12pt]{amsart}
\usepackage{amssymb,bbold}
\usepackage[all]{xy}

\theoremstyle{plain}
\newtheorem{theorem}{Theorem}
\newtheorem{corollary}[theorem]{Corollary}
\newtheorem{lemma}[theorem]{Lemma}
\newtheorem{proposition}[theorem]{Proposition}

\theoremstyle{definition}

\newtheorem{remark}[theorem]{Remark}

\begin{document}
\baselineskip 18pt

\title[Order Continuous and Topological Representations]
      {Order Continuous and Topological Representations of Archimedean Vector Lattices via $S(X)$-spaces}

\author[O.~Zabeti]{Omid Zabeti}
\address[O.~Zabeti]
  {Department of Mathematics, Faculty of Mathematics, Statistics, and Computer science,
   University of Sistan and Baluchestan, Zahedan,
   P.O. Box 98135-674. Iran}
\email{o.zabeti@gmail.com}
\maketitle

\begin{abstract}
For an arbitrary topological space $X$, assume that  $S(X)$ is the vector lattice of all equivalence classes of real-valued continuous functions on open dense subsets of $X$; it is a laterally complete vector lattice but not a normed lattice, certainly. Nevertheless, we can have the extended unbounded norm topology ($un$-topology) on it. On the other hand, by a remarkable result of  Wickstead, there exists a representation approach for every Archimedean vector lattice $E$ in terms of $S(X)$-spaces. In this paper, we show that this representation is order continuous and when $E$ is order complete, it coincides with the known Maeda-Ogasawara representation. Moreover, when $E$ is a Banach lattice, by consideration of the $un$-topology on $E$ and the extended $un$-topology on $S(X)$, we show that this representation is, in fact, a homeomorphism. With the aid of this topological attitude, we establish a representation theorem (in fact a homeomorphism) for the Fremlin projective tensor product between Banach lattices, in terms of $S(X)$-spaces, as well.
\end{abstract}
\date{\today}

\maketitle
\section{Motivation and introduction}

Let us start with some motivation. Assume that $E$ is an Archimedean vector lattice. There are two important and significant representations of $E$ in terms of some functions spaces. The first one is the known  Maeda-Ogasawara representation theorem which states that $E$ can be considered as an order dense vector sublattice of some $C^{\infty}(\Omega)$-space, where $\Omega$ is an extremally disconnected compact Hausdorff space; in fact, $C^{\infty}(\Omega)=E^u$, the universal completion of $E$ (see \cite[Chapter 7]{AB1} for a comprehensive explanation). This approach is practical and useful; however, there is one problem, here. $C^{\infty}(X\times Y)$ may not be a vector space even if $X$ and $Y$ are extremally disconnected compact Hausdorff spaces (see \cite{BW:17} for more details). This problem causes difficulty while we are dealing with tensor products. Let us explain more. By the remarkable Kakutani theroem every Archimedean vector lattice $X$ with an order unit is a norm and order dense sublattice of a $C(K)$-space for some compact Hausdorff space $K$. Furthermore, when $X$ and $Y$ are Archimedean vector lattices with order units and with the representations $C(K_1)$ and $C(K_2)$ ($K_1$ and $K_2$ compact Hausdorff spaces), respectively, the Fremlin tensor product $X\overline{\otimes}Y$ is norm and order dense in  $C(K_1\times K_2)$; which is certainly a vector lattice. Now, if we want to develop such theory for arbitrary vector lattices (with $C^{\infty}(\Omega)$-spaces instead of $C(K)$-spaces), we must overcome the mentioned difficulty.

 Buskes and Wickstead in \cite{BW:17} solved this problem by considering the space $S(X)$ for any topological space $X$. In fact, $S(X)$ consists of all equivalence classes of continuous real-valued functions defined on open dense subsets of $X$; it is an Archimedean vector lattice under pointwise lattice and algebraic operations. In this case, we can have a representation for the Fremlin tensor product of general Archimedean vector lattices in terms of $S(X)$-spaces (for a complete context on this approach, see \cite{BW:17, Z:24}). Recently, Wickstead in \cite{Wickstead:24}, established a representation theorem for Archimedean vector lattices in terms of $S(X)$-spaces. 
 
 In this paper, we show that this representation is order continuous. Furthermore, when $E$ is order complete, we prove that $S(X)$ is order complete and also $X$ is extremally disconnected. Therefore, $S(X)$ and $C^{\infty}(X)$ agree. On the other hand, since $S(X)$ is always laterally complete, it can not be a normed lattice; however, we can have unbounded norm topology ($un$-topology) on it; by the extended $un$-topology considered in \cite{KLT}. These points motivate us to investigate the representation posed by Wickstead for vector lattices, for the case when $E$ is also a Banach lattice.  In fact, we show that when $E$ is an order continuous Banach lattice, there is a representation of $E$ into a $S(X)$-space that is also a homeomorphism (we equip $E$ with the $un$-topology while $S(X)$ enjoys the extended $un$-topology). Furthermore, with using this attitude, we are able to establish a homeomorphism representation for the Fremlin projective tensor product between Banach lattices in terms of $S(X)$-spaces, as well. 
In the sequel, we recall some notes regarding unbounded convergences as well as the Fremlin tensor products between vector and Banach lattices.
\section{preliminaries}
\subsection{unbounded convergences}
Suppose that $E$ is a vector lattice. For a net $(x_{\alpha})$ in $E$, if there is a net $(u_\gamma)$, possibly over a
different index set, with $u_\gamma \downarrow 0$ and for every $\gamma$ there exists $\alpha_0$ such
that $|x_{\alpha} - x| \leq u_\gamma$ whenever $\alpha \geq \alpha_0$, we say that $(x_\alpha)$ converges to $x$ in order, in notation, $x_\alpha \xrightarrow{o}x$. A net $(x_{\alpha})$ in $E$ is said to be unbounded order convergent ($uo$-convergent) to $x \in E$ if for each $u \in E_{+}$, the net $(|x_{\alpha} - x| \wedge u)$ converges to zero in order (for brief, $x_{\alpha}\xrightarrow{uo}x$). For order bounded nets, these notions agree together. For more details on these topics and related notions, see \cite{GTX:17}. Now, assume that $E$ is a Banach lattice. A net $(x_{\alpha})\subseteq E$ is unbounded norm convergent ($un$-convergent) to $x\in E$ provided that for every $u\in E_{+}$, $\||x_{\alpha}-x|\wedge u\|\rightarrow 0$; this convergence is topological; that is $(E,un)$ is a locally solid vector lattice. For more details, see \cite{Den:17, KMT, KLT}. 
Finally, for undefined terminology and general theory of vector lattices and also Banach lattices, we refer the reader to \cite{AB1, AB}.
\subsection{Fremlin tensor product}
In this part, we recall some notes about the Fremlin tensor product between vector lattices and Banach lattices. For more details, see \cite{Fremlin:72, Fremlin:74}. Furthermore, for a comprehensive, new and interesting reference, see \cite{Wickstead1:24}. Moreover, for a short and nicely written exposition on different types of tensor products between Archimedean vector lattices, see \cite{Gr:23}.

Assume that  $E$ and $F$ are Archimedean vector lattices. In 1972, Fremlin constructed a tensor product $E\overline{\otimes} F$ that is an  Archimedean vector lattice such that the algebraic tensor product $E\otimes F$ is a vector subspace of $E\overline{\otimes}F$ so that it is an ordered vector subspace in its own right. Moreover, the vector sublattice in $E\overline{\otimes} F$ generated by $E\otimes F$ is the whole of $E\overline{\otimes}F$. 
Therefore, we conclude that every element of $E\overline{\otimes}F$ can be considered as a finite supremum and finite infimum of some elements of $E\otimes F$. 

 Now, let $E$ and $F$ be Banach lattices. Fremlin in \cite{Fremlin:74} constructed a tensor product $E\widehat{\otimes}F$ that is a Banach lattice. In fact,
 {$E\widehat{\otimes}F$ is the norm completion of $E\otimes F$ with respect to the projective norm: for each $u=\Sigma_{i=1}^{n}x_i\otimes y_i \in E\otimes F$:
 \[\|u\|_{|\pi|}=\sup\{|\Sigma_{i=1}^{n}\phi(x_i,y_i)|: \phi\hspace{0.25cm} \textit{is a bilinear form on}\hspace{0.35cm} E\times F \textit{and} \hspace{0.25cm}\|\phi\|\leq 1\}.\]
 
 Furthermore, $E\overline{\otimes}F$ can be considered as an norm-dense vector sublattice of $E\widehat{\otimes}F$.
 Moreover, the projective norm, $\|.\|_{|\pi|}$, on $E\widehat{\otimes}F$ is a cross norm; that is for every $x\in E$ and for every $y\in F$, we have $\|x\otimes y\|_{|\pi|}=\|x\|\|y\|$. For a comprehensive explanation and also different properties of related to these tensor products, see \cite{Fremlin:72, Fremlin:74}.
 \section{main results}
First, we start with Archimedean vector lattices and $S(X)$-spaces. It is shown in \cite[Proposition 2.5]{Wickstead:24} that $S(X)$-spaces are laterally complete. Now, we proceed to see whether or not this space is order complete.
\begin{lemma}\label{008}
Suppose $X$ is a completely regular topological space. Then $S(X)$ is order complete if and only if $X$ is extremally disconnected.
\end{lemma}
\begin{proof}
Suppose $X$ is extremally disconnected. By \cite[Lemma 1]{Z:24}, $S(X)$ is lattice isomorphic to $C^{\infty}(X)$ so that $S(X)$ is order complete. For the converse, suppose $S(X)$ is order complete. Using \cite[Proposition 2.5]{Wickstead:24}, convinces us that it is universally complete. Therefore, there exists an extremally disconnected compact Hausdorff $Q$ such that $S(X)$ is lattice isomorphic to $C^{\infty}(Q)$; assume that $\phi$ is the desired lattice isomorphism between $S(X)$ and $C^{\infty}(Q)$. On the other hand, by \cite[Lemma 1]{Z:24} again, we can identify $C^{\infty}(Q)$ with $S(Q)$. Note that both $S(X)$ and $S(Q)$ have weak units $\textbf{1}_{X}$ and $\textbf{1}_{Q}$, respectively. Thus, by \cite[Theorem 4.1 and Theorem 4.2]{Wickstead:24}, there is a continuous function $\pi:Q\to X$ such that $\phi(f)=f o\pi$. This shows that the restriction $\phi$ to $C(X)$ maps $C(X)$ into $C(Q)$ and is still a lattice isomorphism. Thus $X$ and $Q$ are homeomorphic so that $X$ is extremally disconnected, as well.
\end{proof}
Suppose $E$ is an Archimedean vector lattice with a weak unit $u$. By \cite[Theorem 3.3]{Wickstead:24}, there exists a compact Hausdorff space $X$ and a lattice isomorphism $T:E\to S(X)$ such that $T(u)={\textbf{1}_{X}}$ and the norm closure (with the sup-norm) of the image of the ideal generated by $u$ in $E$, $E_u$, is $C(X)$. In the following, we show that $T$  preserves  $uo$-convergence.
\begin{lemma}\label{02}
Suppose $E$ is an Archimedean vector lattice $E$ with a weak unit $u$ and $T$ is the representation of $E$ as described in  \cite[Theorem 3.3]{Wickstead:24}. Then $T$ is both $uo$-continuous and order continuous. Moreover, $T^{-1}$ is also $uo$-continuous.
\end{lemma}
\begin{proof}
By \cite[Theorem 5.2]{B:23}, $uo$-continuity and order continuity of $T$ are equivalent. Moreover, suppose $E_u$ is the ideal generated by $u$ in $E$ and $I_{\textbf{1}_{X}}$ is the ideal in $S(X)$ generated by the constant function $\textbf{1}_{X}$. It is easy to see that $T(E_u)$ is a vector sublattice of $I_{\textbf{1}_{X}}$.  Note that $T(E_u)$ is norm dense in $C(X)$ so that it is order dense by \cite[Lemma 1.2]{Fremlin:72}. On the other hand, by \cite[Lemma 5]{Z:24}, $C(X)$ is order dense in $S(X)$. Therefore, we see that $T(E_u)$ is order dense in $S(X)$. This implies that $T(E)$ is order dense in $S(X)$, as well. By considering \cite[Corollary 3.5]{GTX:17} and \cite[Theorem 3.2]{GTX:17}, we conclude that $uo$-continuity in $E$ and $S(X)$ reduces to order continuity of the restriction of $T$ to $E_u$ onto $T(E_u)$. Now, by \cite[Theorem 2.21]{AB}, we have the desired result. Note that by \cite[Theorem 3.2]{GTX:17}, $T(E)$ is a regular Riesz subspace of $S(X)$ so that by \cite[Theorem 3.2]{GTX:17},  $uo$-convergence has the same meaning in both $T(E)$ and $S(X)$. Therefore, $T^{-1}$ is also $uo$-continuous.
\end{proof}
Now, we extend this result to the case when the Archimedean vector lattice does not have weak units. Before that, we have a simple observation.
\begin{remark}\label{0008}
Suppose $(X_{\alpha})$ is a family of topological spaces and $\sqcup_{\alpha}X_{\alpha}$ is the disjoint unions of them. Then, we can identify $S(\sqcup_{\alpha}X_{\alpha})$ and the Cartesian product $\prod_{\alpha}S(X_{\alpha})$, topologically and ordering. More precisely, for each family $(f_{\alpha})\subseteq \prod_{\alpha}S(X_{\alpha})$, there exists a necessarily unique $f\in S(\sqcup_{\alpha}X_{\alpha})$ such that $f=(f_{\alpha})_{\alpha}$. Furthermore, since the ordering on the Cartesian product is assumed to be componentwise, it is easily seen that $|f|=(|f_{\alpha}|)_{\alpha}$. Moreover, it can be verified that if $Y_{\alpha}$ is an order dense vector sublattice of  $S(X_{\alpha})$, then $\prod_{\alpha}Y_{\alpha}$ is an order dense vector sublattice of $\prod_{\alpha}S(X_{\alpha})=S(\sqcup_{\alpha}X_{\alpha})$.
\end{remark}
\begin{theorem}\label{09}
Suppose $E$ is an  Archimedean vector lattice and $T$ is the representation of $E$ as described in \cite[Corollary 3.6]{Wickstead:24}. Then, $T$ is both $uo$-continuous and order continuous. Moreover, $T^{-1}$ is also $uo$-continuous.
\end{theorem}
\begin{proof}
Again by \cite[Theorem 5.2]{B:23}, $uo$-continuity and order continuity of $T$ are equivalent. We show that $T$ is order continuous. We have a decomposition of $E$ into direct sums of pairwise disjoint bands $(B_{\alpha})_{\alpha \in I}$, each of them has a weak unit, namely, $x_{\alpha}$ (by \cite[Theorem 28.5]{LZ}). By Lemma \ref{02}, there are compact Hausdorff spaces $X_{\alpha}$ and order continuous lattice isomorphisms $T_{\alpha}:B_{\alpha} \to S(X_{\alpha})$. Put $X=\sqcup_{\alpha}X_{\alpha}$,  the Hausdorff space of the disjoint unions of $X_{\alpha}'s$. By using \cite[Corollary 3.6]{Wickstead:24}, we see that $T(x)=(T_{\alpha}(y_{\alpha}))_{\alpha}$, in which $x=\vee_{\alpha}y_{\alpha}$ and $y_{\alpha}\in B_{\alpha}$. By \cite[Theorem 3.2]{GTX:17}, it is enough to show that $T(E)$ is a regular vector sublattice of $S(\sqcup_{\alpha}X_{\alpha})$. Note that for each $\alpha$, by Lemma \ref{02}, $T_{\alpha}(B_{\alpha})$ is order dense in $S(X_{\alpha})$. Therefore, form Remark \ref{0008} the conclusion follows.

 \end{proof}
 Among the proof of Theorem \ref{09}, we see that $E$ (identifying with $T(E)$) is an order dense vector sublattice of $S(X)$.  So, by using \cite[Theorem 2.31]{AB}, the following result follows.
 \begin{corollary}\label{007}
 Suppose $E$ is an order complete vector lattice and $T:E\to S(X)$ is its representation as described in Theorem \ref{09}. Then, $E$ can be considered as an ideal of $S(X)$.
 \end{corollary}

 Now, as an application, we present a representation approach for the Fremlin tensor product between Archimedean vector lattices. For details see \cite{Fremlin:72}; for a new and more illustrative approach, see \cite{Wickstead1:24}.
 \begin{corollary}\label{08}
Suppose $E$ and $F$ are Archimedean vector lattices. Moreover, assume that $S(X)$ and $S(Y)$ are the corresponding representations of $E$ and $F$, respectively as described in Theorem \ref{09}. Then, there exists a similar representation for the Fremlin tensor product $E\overline{\otimes}F$ in $S(X\times Y)$, as well.
\end{corollary}
\begin{proof}
By Theorem \ref{09}, there are Hausdorff spaces $X$ and $Y$ and order continuous lattice isomorphisms $T:E\to S(X)$ and $S:F\to S(Y)$. Consider the bi-injective lattice bimorphism defined via $(x,y)\to T(x)\otimes S(y)$ from $E\times F$ into $S(X\times Y)$; it induces a lattice homomorphism $T\otimes S:E\overline{\otimes }F\to S(X\times Y)$  defined via $(T\otimes S)(x\otimes y)=T(x)\otimes S(y)$. We show that $T\otimes S$ is also one-to-one. First, assume that $T(x\otimes y)=T(x)\otimes S(y)=0$. By \cite[Proposition 3.1]{BW:17}, we see that $T(x)=0$ or $S(y)=0$. Therefore, $x=0$ or $y=0$ so that $x\otimes y=0$. Assume that $0\neq u\in (E\overline{\otimes}F)_{+}$ with $(T\otimes S)(u)=0$, by \cite[Theorem 4.2 (4)]{Fremlin:72}, there exist $x_0\in E_{+}$ and $y_0\in F_{+}$ with $0<x_0\otimes y_0\leq u$ so that $T(x_0)\otimes S(y_0)=0$. Therefore, by the previous part, $x_0\otimes y_0=0$ which is a contradiction. Thus, $T(E)$ and $S(F)$ can be considered as order dense vector sublattices of $S(X)$ and $S(Y)$, respectively. By \cite[Lemma 7 and Lemma 8]{Z:24}, we conclude that $T(E)\overline{\otimes}S(F)$ can be considered as an order dense vector sublattice of $S(X\times Y)$. It can be verified that $(T\otimes S)(E\overline{\otimes} F)$ is an order dense vector sublattice of $T(E)\overline{\otimes}S(F)$. Now, \cite[Theorem 3.2]{GTX:17} convinces that $uo$-continuity can be transferred between $E\overline{\otimes}F$ and $S(X\times Y)$ via $T\otimes S$.
\end{proof}
It is known that the universal completion of an Archimedean vector lattice $E$, $E^u$, can be identified with a $C^{\infty}(\Omega)$-space, in which, $\Omega$ is an extremally disconnected compact Hausdorff topological space. On the other hand, we can have the lateral completion of $E$, $E^{\lambda}$, in a similar manner: the intersection of all laterally complete Archimedean vector lattices that contain $E$ as a vector sublattice. By \cite[Page 213, Exercise 10]{AB1}, $(E^{\delta})^{\lambda}=(E^{\lambda})^{\delta}=E^u$, in which, $E^{\delta}$ is the order completion of $E$. It is interesting to know that the lateral completion also has the form of a $S(X)$-space.
\begin{corollary}\label{11}
Suppose $E$ is an order complete vector lattice and $T$ is the lattice isomorphism from $E$ into $S(X)$ as described in \cite[Corollary 3.6]{Wickstead:24}. Then, $E^{\lambda}=E^u=S(X)$ and $X$ is extremally disconnected.
\end{corollary}
\begin{proof}
By Theorem \ref{09}, $T$ is an order continuous lattice isomorphism. Since by \cite[Proposition 2.5]{Wickstead:24}, $S(X)$ is laterally complete and by the assumption, $E$ is order complete, by \cite[Theorem 2.32]{AB}, $T$ can have an extension to an order continuous lattice isomorphism (denoted by $T$, again) from $E^{\lambda}=E^u$ into $S(X)$. By using Corollary \ref{007} and also \cite[Theorem 3.2]{GTX:17}, we can see that $E^u$ is an order dense vector sublattice of $S(X)$. By \cite[Theorem 7.15]{AB1}, $E^{\lambda}=E^u$ and $S(X)$ are in fact the same. Therefore, by  Lemma \ref{008}, $X$ is extremally disconnected.
\end{proof}
Now, we are going to establish the main results of the note. First, we need some preliminaries.

Suppose $X$ is a topological space. Note that $S(X)$ can not be a normed lattice; nevertheless, by the extended $un$-convergence procedure described in \cite{KLT}, we can have $un$-topology on $S(X)$: suppose $Y$ is an Archimedean vector lattice and $E$ is a normed lattice which is an ideal in $Y$, as well. For a net $(y_{\alpha})\subseteq Y$, we say $y_{\alpha}$ is unbounded norm convergent ($un$-convergent) to $y\in Y$ if for every $x\in E_{+}$, $\||y_{\alpha}-y|\wedge x\|\rightarrow 0$. This convergence induces a topology on $Y$ which is called "the extended $un$-topology" on $Y$ induced by $E$. However, this convergence is dependent on the ideal $E$; moreover, it may not be Hausdorff, in general. The good news is that the extended $un$-topology induced by an ideal is Hausdorff if and only if the ideal is order dense (see \cite[Proposition 1.4]{KLT}).

Observe that in general, the extended $un$-convergence is not a linear topology so that we miss many considerable results regarding locally solid vector lattices. However, when $E$ is an order continuous ideal in an Archimedean vector lattice $Y$, then, the extended $un$-topology on $Y$ induced by $E$ is linear  so that $(Y,un-E)$ is a locally solid vector lattice, as well; see \cite[Example 1.5]{KLT} and discussion after that for more details. Now, we present a representation theorem for order continuous Banach lattices.

\begin{theorem}\label{000}
Suppose $E$ is an order continuous Banach lattice. Then, there exists a Hausdorff topological space $X$  and a lattice isomorphism homeomorphism $T:(E, un)\to (S(X), \tau_E)$, in which $S(X)$ is equipped with the extended $un$-topology induced by $E$.
\end{theorem}
\begin{proof}

First, assume that $E$ has a weak unit (quasi-interior point) $u$ and  note that we can identify the image of $E$ under $T$ ($T(E)$) with $E$ (topologically and ordering). Consider order continuous lattice isomorphism $T:E\to S(X)$ as described in Theorem \ref{09}.  $T(E)$ can be considered an an order dense vector sublattice of $S(X)$ so that by \cite[Theorem 2.31]{AB}, it is an ideal. So, we can consider the extended $un$-topology on $S(X)$ induced by $E$. Since $E$ is order continuous, the induced topology on $S(X)$ is linear; that is $S(X)$ is a locally solid vector lattice. By \cite[Theorem 5.19]{AB1}, $T$ is continuous. Furthermore,  the extended $un$-topology is also metrizable by \cite[Theorem 3.3]{KLT}. So, by \cite[Theorem 7.55]{AB}, it is also Lebesgue.


Now, suppose $(f_{\alpha})\subseteq E$ is $un$-null so that $\|f_{\alpha}\wedge u\|\rightarrow 0$. Therefore $\|T(f_{\alpha})\wedge \textbf{1}_{X}\|\rightarrow 0$. Now, \cite[Proposition 3.1]{KLT} implies that $T(f_{\alpha})\xrightarrow{un-E}0$ in $S(X)$.  Now assume that a net $(g_{\alpha})\subseteq T(E)$ is $un$-null. There exists a net $(f_{\alpha})\subseteq E$ with $T(f_{\alpha})=g_{\alpha}$. Since $g_{\alpha}\xrightarrow{un}0$, we conclude that $\|g_{\alpha}\wedge \textbf{1}_{X}\|_{T(E)}=\|T(f_{\alpha}\wedge u)\|_{T(E)}=\|f_{\alpha}\wedge u\|_{E}\rightarrow 0$. Thus, $f_{\alpha}\wedge u\xrightarrow{un}0$ so that $(f_{\alpha})$ is $un$-null since $u$ is a quasi-interior point.

Now, we proceed with the general case. Suppose $E$ is an order continuous Banach lattice. By \cite[Proposition 1.a.9]{LT}, it possesses a dense band decomposition. More precisely, there exists a pairwise disjoint family of bands $\textbf{B}$ such that every band $B_{\alpha}$ in $\textbf{B}$ has a weak unit (quasi-interior point, namely, $x_{\alpha}$) and $E$ is the closure of the direct sums of all elements in $\textbf{B}$. By the former case, there exist compact Hausdorff spaces $X_{\alpha}$ and  lattice isomorphisms $T_{\alpha}:B_{\alpha}\to S(X_{\alpha})$ that are $un$-homeomorphisms. By using \cite[Thoerem 2.14]{AB}, we conclude that for each $B_{\alpha_1}, B_{\alpha_2}\in \textbf{B}$, $T_{\alpha_1}(B_{\alpha_1})\wedge T_{\alpha_2}(B_{\alpha_2})=0$. Now, define the lattice isomorphism $T:E\to S(\sqcup_{\alpha} X_{\alpha})$ defined via $T(x)=T(\Sigma_{\alpha}y_{\alpha})=T(\bigvee_{\alpha }y_{\alpha})=\bigvee_{\alpha}(T_{\alpha}(y_{\alpha}))$.
Note that $S(\sqcup_{\alpha}X_{\alpha})$ can be identified with $\prod_{\alpha}S(X_{\alpha})$ by Remark \ref{0008}.

For each $\alpha$, assume that $P_{\alpha}$ is the natural band projection from $E$ onto $B_{\alpha}$; $B_{\alpha}$ is an order continuous Banach lattice with a quasi-interior point that induces a Hausdorff locally solid Lebesgue topology $\tau_{\alpha}$ on $S(X_{\alpha})$ by the former case. Now, consider the product topology $\tau$ on $S(\sqcup_{\alpha}X_{\alpha})=\prod_{\alpha}S(X_{\alpha})$; it is a Hausdorff locally solid topology by \cite[Theorem 2.20]{AB1} and also Lebesgue by \cite[Theorem 3.11]{AB1}. On the other hand, we have the extended $un$-topology $\tau_E$ on $S(\sqcup_{\alpha}X_{\alpha})$ that is Hausdorff (since $E$ is order dense in $S(X)$), locally solid (since $E$ is order continuous). On the other hand, by Proposition \ref{11} and also by using \cite[Theorem 6.7]{KLT}, $S(\sqcup_{\alpha}X_{\alpha})$ is $un$-complete (with respect to the induced topology $\tau_{E}$). Furthermore, by \cite[Proposition 9.1]{KLT}, $(S(\sqcup_{\alpha}X_{\alpha}),\tau_E)$ satisfies the pre-Lebesgue property so that by \cite[Theorem 3.26]{AB1}, $\tau_E$ is also Lebesgue. Therefore, by \cite[Theorem 7.53]{AB1}, we have $\tau=\tau_E$.

 Suppose $x_{\beta}\xrightarrow{un}0$ in $E$.  By \cite[Theorem 4.12]{KMT}, $P_{\alpha}(x_{\beta})\xrightarrow{un}0$ in $B_{\alpha}$. By the former case, $T_{\alpha}P_{\alpha}(x_{\beta_n})\xrightarrow{un}0$ in $S(X_{\alpha})$. Thus, $T(x_{\beta})=(T_{\alpha}P_{\alpha}(x_{\beta}))_{\alpha}\xrightarrow{\tau}0$.


For the converse, assume that a net $T(x_{\beta})\subseteq T(E)$ is $\tau$-null. Assume that $x_{\beta}=(y^{\beta}_{\alpha})_{\alpha}$, in which $y^{\beta}_{\alpha}\in B_{\alpha}$. Therefore, $T_{\alpha}(y^{\beta}_{\alpha})=T_{\alpha}P_{\alpha}(x_{\beta})\xrightarrow{un}0$ in $S(X_{\alpha})$.  By the former case, $P_{\alpha}(x_{\beta})=y^{\beta}_{\alpha}\xrightarrow{un}0$ in $B_{\alpha}$.
 Now, \cite[Theorem 4.12]{KMT}, convinces us that $x_{\beta}\xrightarrow{un}0$ in $E$ as claimed. 

\end{proof}
As an application, we establish a $un$-homeomorphism representation for the Fremlin projective tensor product of Banach lattices; for more details, see \cite{Fremlin:74}.
Before that, we show that quasi-interior points can be preserved by the Fremlin projective tensor product $E\widehat{\otimes}F$.
\begin{proposition}\label{0000}
Suppose $E$ and $F$ are Banach lattices with quasi-interior points. Then, the Fremlin projective tensor product $E\widehat{\otimes}F$ has a quasi-interior point, as well.
\end{proposition}
\begin{proof}
Assume that $E$ has a quasi-interior $x_0$ and $F$ possesses a quasi-interior point $y_0$. We show that $x_0\otimes y_0$ is a quasi-interior point for  $E\widehat{\otimes}F$. By \cite[Theorem 4.85]{AB}, it is enough to show that for every $0<f\in  (E\widehat{\otimes}F)'$, $f(x_0\otimes y_0)>0$. By \cite[1(A) d]{Fremlin:74}, $(E\widehat{\otimes}F)'=B^{r}(E\times F)$ where $B^{r}(E\times F)$ is the Banach lattice of all bounded regular bilinear forms on $E\times F$. There exist $0 \neq x_1\in E_{+}$ and $0 \neq y_1 \in F$ with $f(x_1,y_1)> 0$ so that $f(x_1,y_0)\neq 0$ since the restriction $f$ to $x_1$, is a non-zero positive functional on $Y$. On a contrary, assume that $f(x_0,y_0)=0$. Note that $x_1\wedge n x_0\rightarrow x_1$ so that $f(x_1\wedge n x_0, y_0)\rightarrow f(x_1,y_0)$, since $f$ is continuous. Note that $f(x_1\wedge n x_0,y_0)\leq n f(x_0,y_0)=0$ which is a contradiction.
\end{proof}
Suppose $E$ is a Banach lattice. Recall that $E$ possesses a dense band decomposition if there exists a family  $\textbf{B}$ of pairwise disjoint projection bands in $E$ such that the linear span of all of the bands in $\textbf{B}$ is norm dense in $E$. For more details see \cite[Section 4.1]{KMT}. For example by \cite[Proposition 1.a.9]{LT}, every order continuous Banach lattice possesses a dense band decomposition; see also \cite[Theorem 4.11]{KMT}. In the following, we show that if Banach lattices $E$ and $F$ have dense band decompositions, then, so is the Fremlin projective tensor product $E\widehat{\otimes}F$.
\begin{lemma}\label{90}
Suppose $E$ and $F$ are Banach lattices such that the Fremlin tensor product $E\overline{\otimes}F$ is order complete. Moreover, assume that $\textbf{B}=(B_{\alpha})_{\alpha \in I}$ and $\textbf{C}=(C_{\beta})_{\beta \in J}$ are dense band decompositions $E$ and $F$, respectively. Then, the collection $\textbf{A}=\{B_{\alpha}\widehat{\otimes}C_{\beta}; B_{\alpha}\in \textbf{B}, C_{\beta}\in \textbf{C}\}$ forms a dense band decomposition for $E\widehat{\otimes}F$.
\end{lemma}
\begin{proof}
First, observe that by \cite[Proposition 3.9]{Gr:23}, both $E$ and $F$ are order complete. Put ${\textbf{A}_{0}}=\{B_{\alpha}\overline{\otimes}C_{\beta}; B_{\alpha}\in \textbf{B}, C_{\beta}\in \textbf{C}\}$. By \cite[Theorem 5.8]{Amor:22}, we see that each element of ${\textbf{A}_{0}}$ is a projection band in $E\overline{\otimes}F$. By \cite[Theorem 2.48]{AB1}, the elements of ${\textbf{A}}$ are also projection bands in $E\widehat{\otimes}F$. Moreover, the elements of ${\textbf{A}}$ are pairwise disjoint. We claim that ${\textbf{A}}$ is a dense band decomposition for $E\widehat{\otimes}F$. We use \cite[Lemma 4.10]{KMT}. Note that the elements of $\textbf{A}$ are pairwise disjoint. For each $\alpha,\alpha',\beta,\beta'$, we have $(B_{\alpha}\otimes C_{\beta})\wedge (B_{\alpha'}\otimes C_{\beta'})=0$. Otherwise, for each non-zero positive $u\in B_{\alpha}\otimes C_{\beta})\wedge (B_{\alpha'}\otimes C_{\beta'})$, by \cite[1(A) d]{Fremlin:74}, we can find $x_0\in {B_{\alpha}}_{+}$, $x_1\in{B_{\alpha'}}_{+}$, $y_0\in{C_{\beta}}_{+}$ and $y_1\in{C_{\beta'}}_{+}$ with $u\leq x_0\otimes y_0$ and $u\leq x_1\otimes y_1$. So,
\[u\leq (x_0\otimes y_0)\wedge (x_1\otimes y_1)\leq (x_0\wedge x_1)\otimes (y_0\vee y_1)=0,\]
that is a contradiction. Now, it is routine to check that if for two subsets $A,B$ in a normed lattice $A\wedge B=0$, then, $\overline{A}\wedge\overline{B}=0$. So, the elements of $\textbf{A}$ are pairwise disjoint.
We need to characterize band projections for elements of ${\textbf{A}}$. Fix $\alpha \in I$ and $\beta \in J$. Assume that $S_{\alpha,\beta}$ is the corresponding band projection from $E\widehat{\otimes}F$ onto $B_{\alpha}\widehat{\otimes}C_{\beta}$. Also, assume that $P_{\alpha}$ and $Q_{\beta}$ are corresponding band projections onto $B_{\alpha}$ and $C_{\beta}$, respectively. By considering the lattice bimorphism $\sigma: E\times F \to  B_{\alpha}\overline{\otimes}C_{\beta}$ defined by $\sigma(x,y)=P_{\alpha}(x)\otimes Q_{\beta}(y)$ and using \cite[1A(b)]{Fremlin:74}, there exists a lattice homomorphism $P_{\alpha}\otimes Q_{\beta}:E\overline{\otimes}F \to B_{\alpha}\overline{\otimes}C_{\beta}$ via $(P_{\alpha}\otimes Q_{\beta})(x\otimes y)=P_{\alpha}(x)\otimes Q_{\beta}(y)$. We claim that $P_{\alpha}\otimes Q_{\beta}=S_{\alpha,\beta}$ on $E\overline{\otimes}F$  and so on $E\widehat{\otimes}F$ by taking a norm completion.

Assume that $x\in E$ and $y\in F$. We can write $x\otimes y=u_{\alpha,\beta}+v_{\alpha,\beta}$ in which, $u_{\alpha,\beta}\in B_{\alpha}\overline{\otimes}C_{\beta}$ and $v_{\alpha,\beta}\in {(B_{\alpha}\overline{\otimes}C_{\beta})}^{d}$ and this representation is unique.
On the other hand, we can also write $x=r_{\alpha}+{r_{\alpha}}^{d}$ and $y=w_{\beta}+{w_{\beta}}^{d}$, in which $r_{\alpha}\in B_{\alpha}$, ${r_{\alpha}}^{d}\in {B_{\alpha}}^{d}$, $w_{\beta}\in C_{\beta}$ and ${w_{\beta}}^{d}\in {C_{\beta}}^{d}$. Therefore, we have
\[x\otimes y=(r_{\alpha}+{r_{\alpha}}^{d})\otimes (w_{\beta}+{w_{\beta}}^{d})=r_{\alpha}\otimes w_{\beta}+{r_{\alpha}}^{d}\otimes w_{\beta}+r_{\alpha}\otimes {w_{\beta}}^{d}+{r_{\alpha}}^{d}\otimes {w_{\beta}}^{d}.\]
Note that $r_{\alpha}\otimes w_{\beta}\in B_{\alpha}\overline{\otimes}C_{\beta}$ and
${r_{\alpha}}^{d}\otimes w_{\beta}+r_{\alpha}\otimes {w_{\beta}}^{d}+{r_{\alpha}}^{d}\otimes {w_{\beta}}^{d}\in (B_{\alpha}\overline{\otimes}C_{\beta})^{d}$. Thus,
$u_{\alpha,\beta}=r_{\alpha}\otimes w_{\beta}$ by uniqueness of the representation. So, $S_{\alpha,\beta}(x\otimes y)=u_{\alpha,\beta}=r_{\alpha}\otimes w_{\beta}=P_{\alpha}(x)\otimes Q_{\beta}(y)$. Therefore, $P_{\alpha}\otimes Q_{\beta}=S_{\alpha,\beta}$ on $E\otimes F$. Every element of $E\overline{\otimes}F$ is a finite suprema and a finite infima of some elements of $E\otimes F$. Since $S_{\alpha,\beta}$ is a band projection, it is order continuous lattice homomorphism so that $P_{\alpha}\otimes Q_{\beta}=S_{\alpha,\beta}$ on $E\overline{\otimes}F$ and so on $E\widehat{\otimes}F$ by an extension that is also a band projection, as well.

Now, suppose $x\in E$ and $y\in F$ and also $\varepsilon>0$ is arbitrary. By \cite[Lemma 4.10]{KMT}, we can find indices $\{\alpha_1,\ldots,\alpha_r\}$ and also $\{\beta_1\ldots,\beta_s\}$ such that $\|x-\bigvee_{i=1}^{r}P_{\alpha_i}(x)\|<\frac{\varepsilon}{2\|y\|}$ and $\|y-\bigvee_{j=1}^{s}Q_{\beta_j}(y)\|<\frac{\varepsilon}{2\|x\|}$. Note that $\||\bigvee_{i=1}^{r}P_{\alpha_i}(x)|\|\leq \|\bigvee_{i=1}^{r}|P_{\alpha_i}(x)|\|=\|\bigvee_{i=1}^{r}P_{\alpha_i}(|x|)\|\leq \|x\|$. Therefore,
\[ \|x\otimes y-\bigvee_{i=1}^{r} \bigvee_{j=1}^{s} P_{\alpha_i}(x)\otimes Q_{\beta_j}(y)\|=
\|x\otimes y-\bigvee_{i=1}^{r} \bigvee_{j=1}^{s} P_{\alpha_i}(x)\otimes Q_{\beta_j} +\bigvee_{i=1}^{r} P_{\alpha_i}(x)\otimes y- \bigvee_{i=1}^{r}P_{\alpha_i}(x)\otimes y\|\leq\]
\[\|x- \bigvee_{i=1}^{r}P_{\alpha_i}(x)\|\|y\|+
\||\bigvee_{i=1}^{r}P_{\alpha_i}(x)|\|\|\bigvee_{j=1}^{s}(y-Q_{\beta_j}(y))\|<\varepsilon.\]
So, we conclude that for each $v\in E\otimes F$, we can find indices $\{\alpha_1,\ldots,\alpha_n\}$ and $\{\beta_1,\ldots, \beta_m\}$ such that  $\|v-\bigvee_{i=1}^{n}\bigvee_{j=1}^{m}(P_{\alpha_i}\otimes Q_{\beta_j})(v)\|<\frac{\varepsilon}{2}$.

Now, we show that for each $u\in E\widehat{\otimes}F$, we have $\|u-\bigvee_{i=1}^{n}\bigvee_{j=1}^{m}(P_{\alpha_i}\otimes Q_{\beta_j})(u)\|<\varepsilon$. This completes the proof.
By density, there exists $v\in E\otimes F$ with $\|u-v\|<\frac{\varepsilon}{2}$. By the former case, $\|v-\bigvee_{i=1}^{n}\bigvee_{j=1}^{m}(P_{\alpha_i}\otimes Q_{\beta_j})(v)\|<\frac{\varepsilon}{2}$. We have
\[\|u-\bigvee_{i=1}^{n}\bigvee_{j=1}^{m}(P_{\alpha_i}\otimes Q_{\beta_j})(u)\|\leq \|(u-v)-\bigvee_{i=1}^{n}\bigvee_{j=1}^{m}(P_{\alpha_i}\otimes Q_{\beta_j})(u-v)\|+ \|v-\bigvee_{i=1}^{n}\bigvee_{j=1}^{m}(P_{\alpha_i}\otimes Q_{\beta_j})(v)\|<\varepsilon.\]
\end{proof}
Note that an order continuous normed lattice need not be order complete. For example the normed lattice consisting of all step-functions in $L^2[0,1]$ is order continuous but not order complete. So, order continuity of $E\overline{\otimes} F$ does not imply order completeness of it, in general. Nevertheless, we can have a similar version of Lemma \ref{90} in this setting.
\begin{lemma}\label{91}
Suppose $E$ and $F$ are Banach lattices such that the Fremlin tensor product $E\overline{\otimes}F$ is order continuous. Moreover, assume that $\textbf{B}=(B_{\alpha})_{\alpha \in I}$ and $\textbf{C}=(C_{\beta})_{\beta \in J}$ are dense band decompositions of $E$ and $F$, respectively. Then, the collection $\textbf{A}=\{B_{\alpha}\widehat{\otimes}C_{\beta}; B_{\alpha}\in \textbf{B}, C_{\beta}\in \textbf{C}\}$ forms a dense band decomposition for $E\widehat{\otimes}F$.
\end{lemma}
\begin{proof}
The proof  essentially has the same idea as the proof of Lemma \ref{90}. First, note that by \cite[Theorem 3.27]{AB1}, we conclude that $E\widehat{\otimes}F$ is also order continuous so that both $E$ and $F$ are order continuous by \cite[Lemma 12]{Z:24}. Therefore, every band in $E$, $F$ and $E\widehat{\otimes}F$ is a projection band. For each $\alpha \in I$ and for each $\beta \in J$, assume that $B_{\alpha}=B_{x_{\alpha}}$ and $C_{\beta}=C_{y_{\beta}}$. By \cite[Theorem 4.2]{Amor:22}, $B_{\alpha}\overline{\otimes}C_{\beta}$ is an order dense vector sublattice in the band in $E\overline{\otimes} F$ generated by $x_{\alpha}\otimes y_{\beta}$, denoted by ${D}_{\alpha,\beta}$, so that it is norm dense. On the other hand, every band is norm closed so that ${D}_{\alpha,\beta}$ is also a band (closed ideal) in $E\widehat{\otimes}F$ by \cite[Theorem 3.8]{AB1}. Therefore, $\textbf{A}$ consisting of projection bands in $E\widehat{\otimes}F$. The rest of the proof is similar to the proof of Lemma \ref{90}.
\end{proof}
\begin{theorem}\label{009}
Suppose $E$ and $F$ are Banach lattices such that $E\overline{\otimes}F$ is both order continuous and order complete. Moreover, assume that $S(X)$ and $S(Y)$ are the corresponding representations of $E$ and $F$, respectively as described in Theorem \ref{000}. Then, there exists a similar representation for the Fremlin projective tensor product of $E$ and $F$ in $S(X\times Y)$, as well.
\end{theorem}
\begin{proof}
First, note that order continuity of $E\overline{\otimes}F$ implies order continuity of $E\hat{\otimes}F$ so that order continuity of both $E$ and $F$ by \cite[Lemma 12]{Z:24}. By considering Theorem \ref{000},  there are Hausdorff spaces $X$ and $Y$, lattice isomorphisms and also $un$-homeomorphisms $T:E\to S(X)$ and $S:F\to S(Y)$. Consider the bi-injective lattice bimorphism defined via $(x,y)\to T(x)\otimes S(y)$ from $E\times F$ into $S(X\times Y)$; it induces a lattice isomorphism $T\otimes S:E\overline{\otimes }F\to S(X\times Y)$  defined via $(T\otimes S)(x\otimes y)=T(x)\otimes S(y)$. By the assumption, $E\overline{\otimes}F$ is order complete. So,by \cite[Theorem 4.31]{AB1}, it is order dense in $E\widehat{\otimes}F$. Therefore, $T\otimes S$ can have a unique order continuous extension from $E\widehat{\otimes}F$ into $S(X,\times Y)$ by \cite[Theorem 2.32]{AB}. Therefore, $E\widehat{\otimes}F$ is an order complete order dense vector sublattice of $S(X\times Y)$ so that  an ideal by \cite[Theorem 2.31]{AB}. So, we can have the extended $un$-topology on $S(X\times Y)$ that make it locally solid because of order continuity of $E\widehat{\otimes}F$.

Note that the extended $un$-topology has the $\sigma$-Lebesgue property by \cite[Theorem 7.49]{AB1}. Also, by \cite[Proposition 9.1]{KLT}, it satisfies the pre-Lebesgue property. So, it has a Lebesgue property by \cite[Theorem 3.27]{AB1}. We show that $T\otimes S$ is a $un$-homeomorphism. Suppose $(u_{\alpha})\subseteq (E\widehat{\otimes} F)_{+}$ is $un$-null. There exists an increasing sequence $(\alpha_n)$ of indices such that $u_{\alpha_n}\xrightarrow{un}0$ and $u_{\alpha_n}\xrightarrow{uo}0$ as well by \cite[Corollary 3.5]{Den:17}. By Corollary \ref{08}, $(T\otimes S)(u_{\alpha_n})\xrightarrow{uo}0$ in $S(X\times Y)$. By \cite[Proposition 9.2]{KLT}, we see that $(T\otimes S)(u_{\alpha_n})\xrightarrow{un}0$. Since $un$-convergence is topological, we conclude that $(T\otimes S)(u_\alpha)\xrightarrow{un}0$.

 For the converse, assume that $(T\otimes S)(x_{\gamma})\xrightarrow{un}0$. First, assume that both $E$ and $F$ have quasi-interior points so that by Proposition \ref{0000}, $E\widehat{\otimes}F$ has a quasi-interior point, as well. Moreover, the corresponding topological spaces $X$ and $Y$ can be assumed to be compact. By \cite[Corollary 3.2]{KLT}, there exists an increasing sequence $(\gamma_n)$ of indices such that that $(T\otimes S)(x_{\gamma_n})\xrightarrow{un}0$. By \cite[Theorem 9.5]{KLT} and by passing to a further subsequence, we may assume that $(T\otimes S)(x_{\gamma_n})\xrightarrow{uo}0$ in $S(X\times Y)$. By using regularity of $E\widehat{\otimes}F$ in $S(X\times Y)$ and also Lemma \ref{02}, $x_{\gamma_n}\xrightarrow{uo}0$ in $E\widehat{\otimes}F$ so that $x_{\gamma_n}\xrightarrow{un}0$. Again, since $un$-convergence is topological, we see that $x_{\gamma}\xrightarrow{un}0$. For the general case, we may use the dense band decomposition for $E\widehat{\otimes}F$ as described in  Lemma \ref{91}. Note that since $E$ and $F$ are order continuous, by \cite[Proposition 1.a.9]{LT}, they possesses  dense band decompositions $\textbf{B}=(B_{\alpha})_{\alpha \in I}$ and $\textbf{C}=(C_{\beta})_{\beta \in J}$, respectively. By Lemma \ref{91}, $\textbf{A}=\{B_{\alpha}\widehat{\otimes}C_{\beta}; B_{\alpha}\in \textbf{B}, C_{\beta}\in \textbf{C}\}$ forms a dense band decomposition for $E\widehat{\otimes}F$. By the former case, there are compact Hausdorff spaces $(X_{\alpha})_{\alpha\in I}$ and $(Y_{\beta})_{\beta \in J}$ and also lattice isomorphisms $T_{\alpha}\otimes S_{\beta}:B_{\alpha}\widehat{\otimes}C_{\beta}\to S(X_{\alpha}\otimes Y_{\beta})$ that are $un$-homeomorphisms. Now, we can use from a similar representation as we had for Theorem \ref{09}.


  Assume that $(T\otimes S)(x_{\gamma})\xrightarrow{un}0$. Write $x_{\gamma}=(y_{\gamma}^{\alpha,\beta})$ in which $y_{\gamma}^{\alpha,\beta}\in B_{\alpha}\widehat{\otimes}C_{\beta}$. Therefore, $T_{\alpha}\otimes S_{\beta}(y_{\gamma}^{\alpha,\beta})\xrightarrow{un}0$ in $S(X_{\alpha}\times Y_{\beta})$. By the former case, $(P_{\alpha}\otimes Q_{\beta})(x_{\gamma})=y_{\gamma}^{\alpha,\beta}\xrightarrow{un} 0$ in $B_{\alpha}\widehat{\otimes}C_{\beta}$. Now, by using \cite[Theorem 4.12]{KMT}, we see that $x_{\gamma}\xrightarrow{un}0$ in $E\widehat{\otimes}F$.

\end{proof}

{\bf On behalf of all authors, the corresponding author states that there is no conflict of interest.}

\end{document}